\documentclass{amsart}

\usepackage{amssymb,tikz}
\usepackage{amsthm}
\usepackage{amsmath, a4wide}
\usepackage{amsbsy}
\usepackage[all]{xy}
\usepackage{bm}
\usepackage{hyperref}
\usepackage{color}
\usepackage{subfigure}
\usepackage{tikz}

\begin{document}

\newtheorem{thm}{Theorem}[section]
\newtheorem{cor}[thm]{Corollary}
\newtheorem{qu}[thm]{Question}
\newtheorem{pr}[thm]{Problem}
\newtheorem{lem}[thm]{Lemma}
\newtheorem{defi}[thm]{Definition}
\newtheorem{prop}[thm]{Proposition}
\newtheorem{opg}[thm]{Opgave}
\theoremstyle{definition}
\newtheorem{rem}[thm]{Remark}
\newtheorem{ex}[thm]{Example}
\newtheorem{cond}[thm]{Condition}
\newtheorem*{pro}{Problem}

\newenvironment{Proof}{\begin{trivlist}

\item[\hskip\labelsep{\it Proof.}]}{$\hfill\Box$\end{trivlist}}

\newcommand{\NN}{\mathbb{N}}
\newcommand{\ZZ}{\mathbb{Z}}
\newcommand{\RR}{\mathbb{R}}
\newcommand{\EE}{\mathbb{E}}
\newcommand{\PP}{\mathbb{P}}
\newcommand{\Ss}{\mathbb{S}}
\newcommand{\QQ}{\mathbb{Q}}

\newcommand{\BB}{\mathbb{B}}

\newcommand{\Pset}{\mathcal{P}}
\newcommand{\Gset}{\mathcal{G}}

\newcommand{\Mlattice} {L_{N,a,b}}
\newcommand{\MMlattice} {L_{N,1,b}}
\newcommand{\MlatticeQ} {L_{N_i,p_i,r_i}}
\newcommand{\MlatticeP}[3] {L_{#1,#2,#3}}
\newcommand{\Kset} {K_{\alpha,\beta}(N)}
\newcommand{\KsetQ} {K_{\alpha,\beta}(N_i)}

\newcommand{\LL}{\mathcal{L}}

\title[TSP of 2-dimensional Kronecker sets]{Bounds for the traveling salesman paths of two-dimensional modular lattices}

\author{Florian Pausinger}
\address{TU M\"{u}nchen, Zentrum Mathematik, M10 Lehrstuhl f\"{u}r Geometrie und Visualisierung }
\email{florian.pausinger@ma.tum.de}

\maketitle

\begin{abstract} 
We present tight upper and lower bounds for the traveling salesman path through the points of two-dimensional modular lattices. We use these results to bound the traveling salesman path of two-dimensional Kronecker point sets. Our results rely on earlier work on shortest vectors in lattices as well as on the strong convergence of Jacobi-Perron type algorithms.
\end{abstract}

\section{Introduction}
\label{sec1}

\subsection{Traveling Salesman Problem}
The traveling salesman problem asks for the length\\ $\LL(x_1, \ldots, x_N)$ of the shortest path through the points $\{ x_1, \ldots, x_N\} \subset \RR^d$. Setting $x_{\sigma(N+1)}:=x_{\sigma(1)}$, we write
$$ \LL(x_1, \ldots, x_N) = \min_{\sigma} \sum_{n=1}^{N} \| x_{\sigma(n)} - x_{\sigma(n+1)} \|, $$
where the minimum is over all permutations $\sigma$ of $\{ 1,2,\ldots, N \}$ and $\| v \|$ be the $2$-norm of a vector $v$.
A theorem of Beardwood, Halton and Hammersley \cite{BHH59}  gives precise asymptotic results for the case of uniformly distributed random variables in $[0,1]^d$. 
\begin{thm}[Beardwood, Halton, Hammersley, 1959] \label{thm:bhh} Let $X_1, X_2, \ldots, X_N$ be i.i.d. uniformly distributed random variables in $[0,1]^d$. Then there exists a constant $\beta(d)$ such that
\begin{equation*}
\underset{N \rightarrow \infty} \lim \frac{\LL(X_1,\ldots, X_N)}{N^{(d-1)/d}} = \beta(d)
\end{equation*}
with probability 1. 
\end{thm}
This classical theorem is impressively contrasted by a recent result of Arlotto and Steele \cite{AS15} who construct a stationary ergodic process $X_1, X_2, \ldots $ such that each $X_t$ has the uniform distribution on the unit square and the length of the shortest path is shown \emph{not} to be asymptotic to a constant times the square root of the number of points.\\
Returning to Theorem \ref{thm:bhh}, it is interesting that the exact value of the constant remains unknown despite serious efforts; see \cite{Ste15} for an overview.
For $d=2$, the bounds $0.625\leq \beta(2) \leq 0.922$ obtained in \cite{BHH59} were recently slightly improved by Steinerberger \cite{Ste15}.
Thus, it is natural to ask in which cases, i.e. for which point sets, it is actually possible to explicitly determine the length of the traveling salesman path.
We observe that the length of the shortest path through $N=n^2$ points arranged on a regular grid $G_N \subset [0,1]^2$ is roughly $\sqrt{N}$. In this case the asymptotic constant is 1. To see this, note that the shortest distance between two neighboring points is $1/(n-1)$ and the $n^2$ points are arranged in $n$ parallel lines. Each line contains $n$ points and, thus, has length 1. Two neighboring lines can be connected with a line segment of length $1/(n-1)$, whereas the first and last line are connected by a line segment of length $\leq \sqrt{2}$. Hence, $\sqrt{N} \leq \LL(G_N) \leq \sqrt{N} + 1 + \sqrt{2}$.\\

\subsection{Preliminaries} In this note we study the traveling salesman path (TSP) through the points of two-dimensional modular lattices as well as two-dimensional Kronecker point sets. 
A two-dimensional \emph{lattice} $L(b_1, b_2)$ is the set of all integer linear combinations of two linearly independent vectors $b_1, b_2$, which are said to generate the lattice. 
We define the \emph{length of the shortest non-zero vector} in the lattice $L$ as
\begin{equation*}
\lambda(L):= \min_{v \in L\setminus \{ 0\} } \| v\|.
\end{equation*}

\begin{center}
\begin{figure}[h!]
\subfigure{
\includegraphics[width=0.25\textwidth]{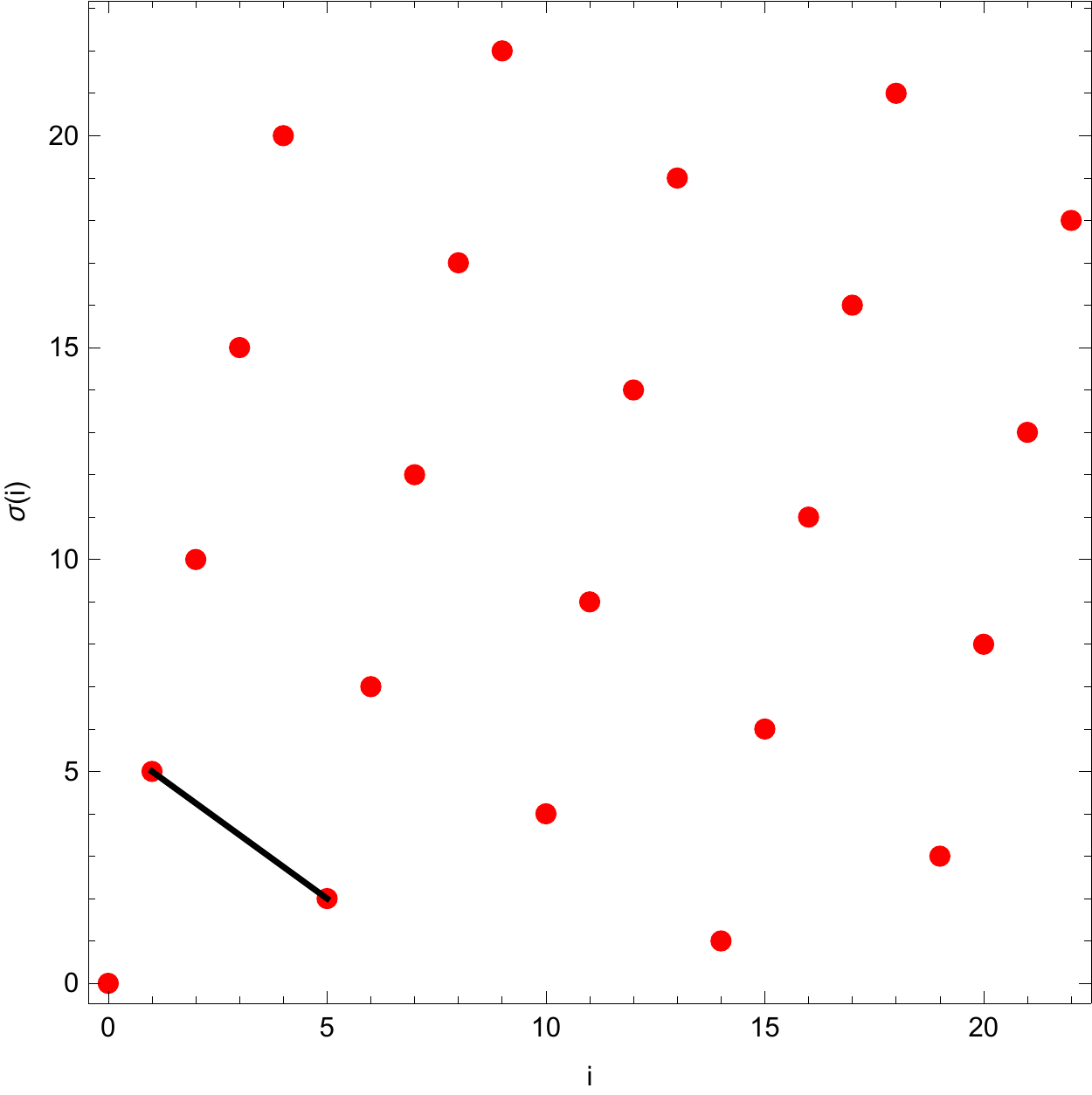}
} \ \ \ \ \ 
\subfigure{
\includegraphics[width=0.25\textwidth]{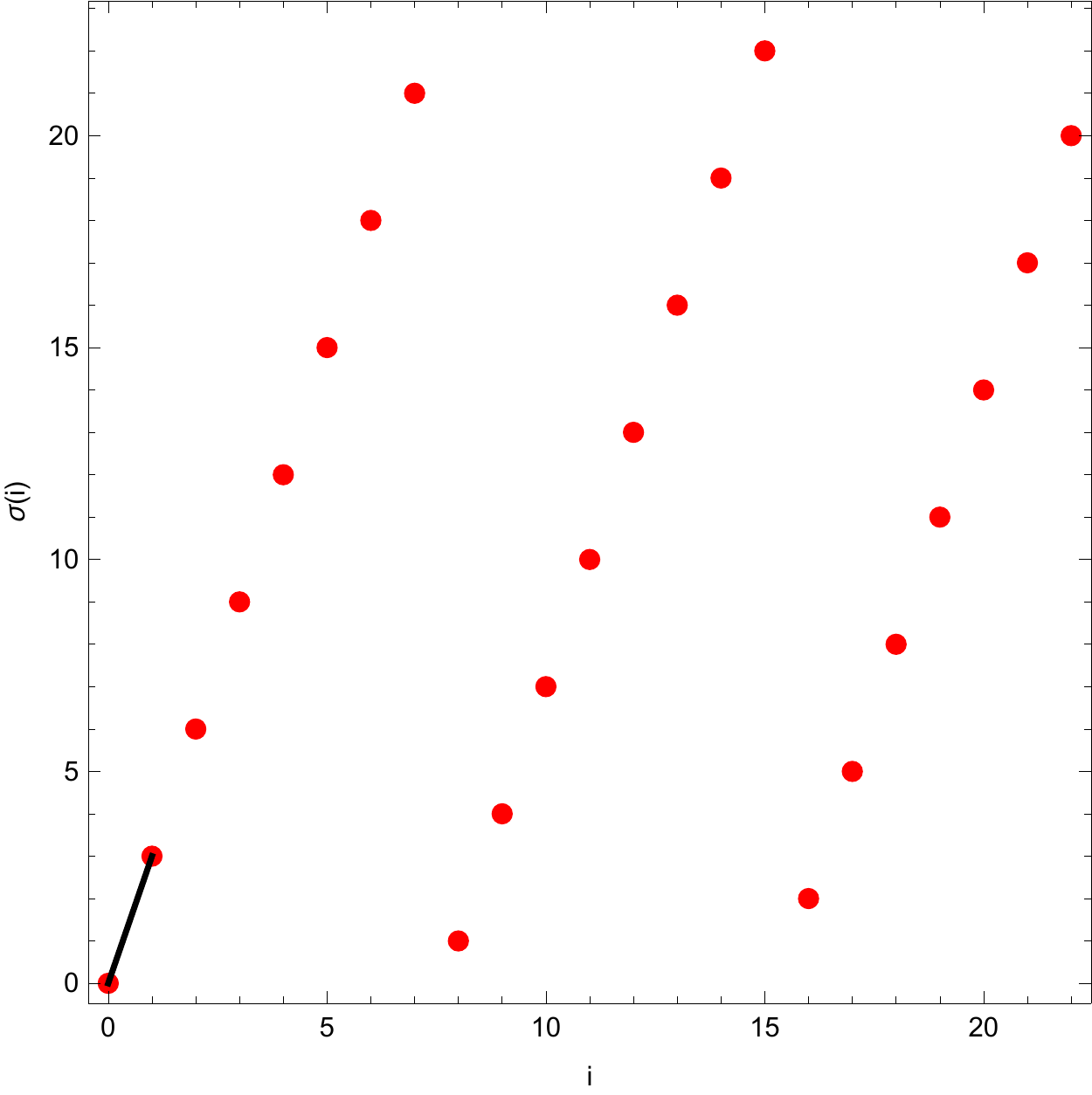}
}
\caption{Illustration of the \emph{shortest distance} in a modular lattice. The modular lattice $L_{23,1,5}$ (\emph{left}) does not contain the shortest vector of the corresponding lattice $L((1,5),(23,0),(0,23)) \subset \RR^2$. The lattice $L_{23,1,3}$ (\emph{right}) contains the shortest vector.} \label{fig:shortest}
\end{figure}
\end{center}
Let $0< a,b < N$ be integers. We define the \emph{modular lattice} $\Mlattice$ generated by the pair $(a,b)$ as 
$$\Mlattice:=\{ (na \! \! \! \! \pmod{N}, nb \! \! \! \! \pmod{N}) : 0 \leq n < N\}.$$ 
Thus, $\Mlattice$ is a subset of the square $[0,N-1]^2$.
Obviously, the set $\Mlattice$ is essentially the same as the lattice $L=L((a,b),(N,0),(0,N))$ generated by the vectors $(a,b),(N,0),(0,N)$:
If we restrict $L$ to $[0,N-1]^2$ we get $\Mlattice$, and if we add all integer multiples of $(N,0),(0,N)$ to $\Mlattice$ we obtain $L$.
Importantly, the shortest vector of $L$ is not necessarily contained in $\Mlattice$. However, if we are interested in the \emph{shortest distance} between points of $\Mlattice$, then this distance is exactly given by the length of the shortest vector of $L$; see Figure \ref{fig:shortest}.
In the following, whenever we work with modular lattices, we will abuse notation and write $\lambda(\Mlattice)$ for the shortest distance between points in $\Mlattice$, which is the length of the shortest vector of the corresponding lattice $L$.
For every modular lattice $\Mlattice$ with $\mathrm{gcd}(N,a,b)=1$ and arbitrary, distinct points $X=(na,nb), Y=(ma,mb)$, we have that $\| (X-Y) \pmod{N} \| \geq \sqrt{2}$. Hence, together with the upper bound from Lemma \ref{lem:help}, we have that 
$$\sqrt{2} \leq \lambda(\Mlattice) \leq 3/2 \sqrt{N}. $$
Moreover, let $0<\alpha, \beta<1$ be irrationals which are, together with 1, linearly independent over $\QQ$. We define the two-dimensional \emph{Kronecker point set}, $\Kset$, as the set of vectors 
$$\Kset:=\{ (n\alpha \! \! \! \! \pmod{1}, n \beta \! \! \! \! \pmod{1}) : 0 \leq n < N\},$$
which is a subset of $[0,1]^2$.
Both families of point sets are quasi random point sets. This means they are (for the right choice of parameters) very uniformly distributed over the unit square $[0,1]^2$ and, thus, often preferred over random points as integration nodes in numerical integration; see Figure \ref{fig:pointsets}.

\begin{center}
\begin{figure}[h!]
\subfigure{
\includegraphics[width=0.31\textwidth]{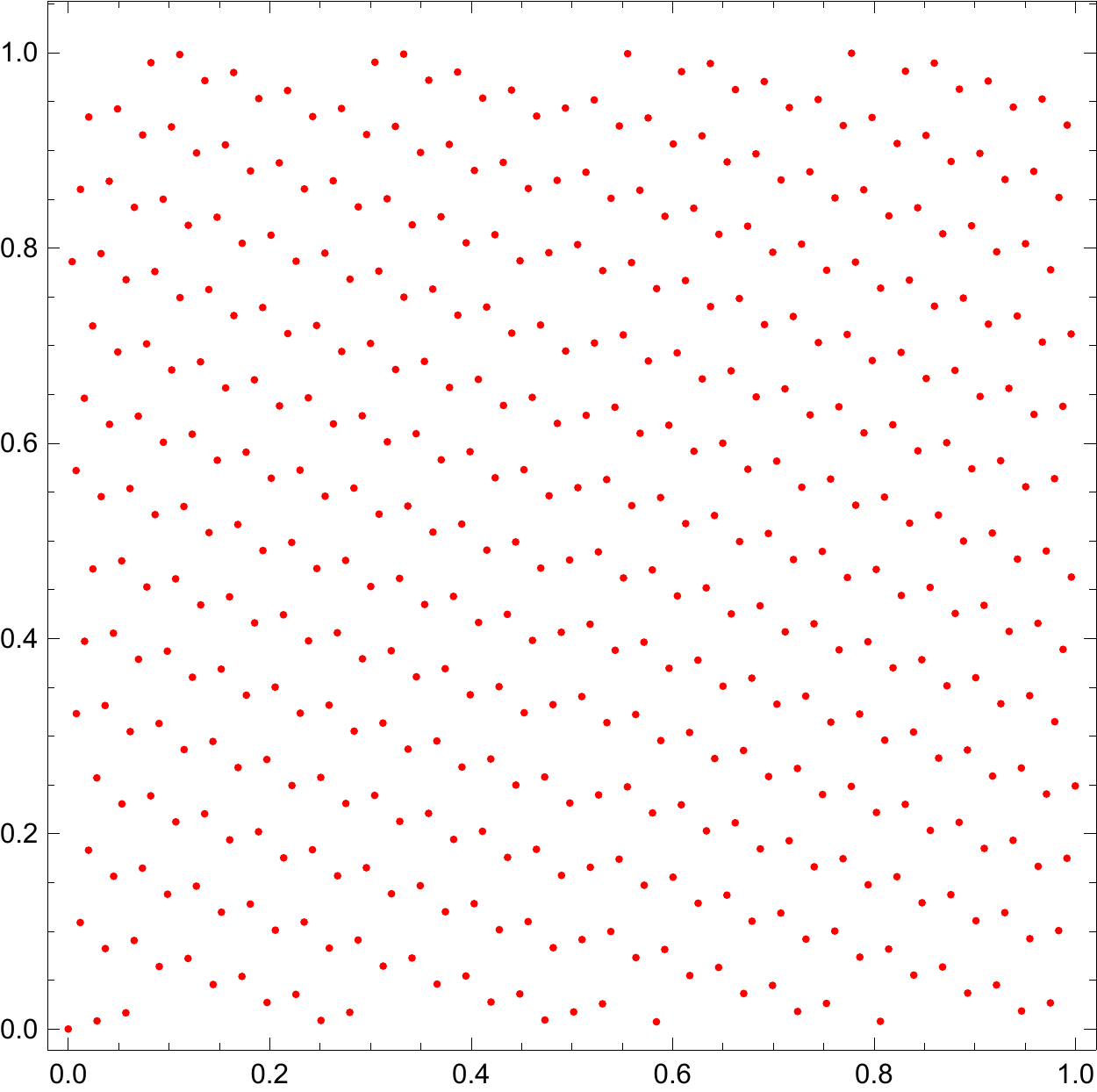}
}
\subfigure{
\includegraphics[width=0.31\textwidth]{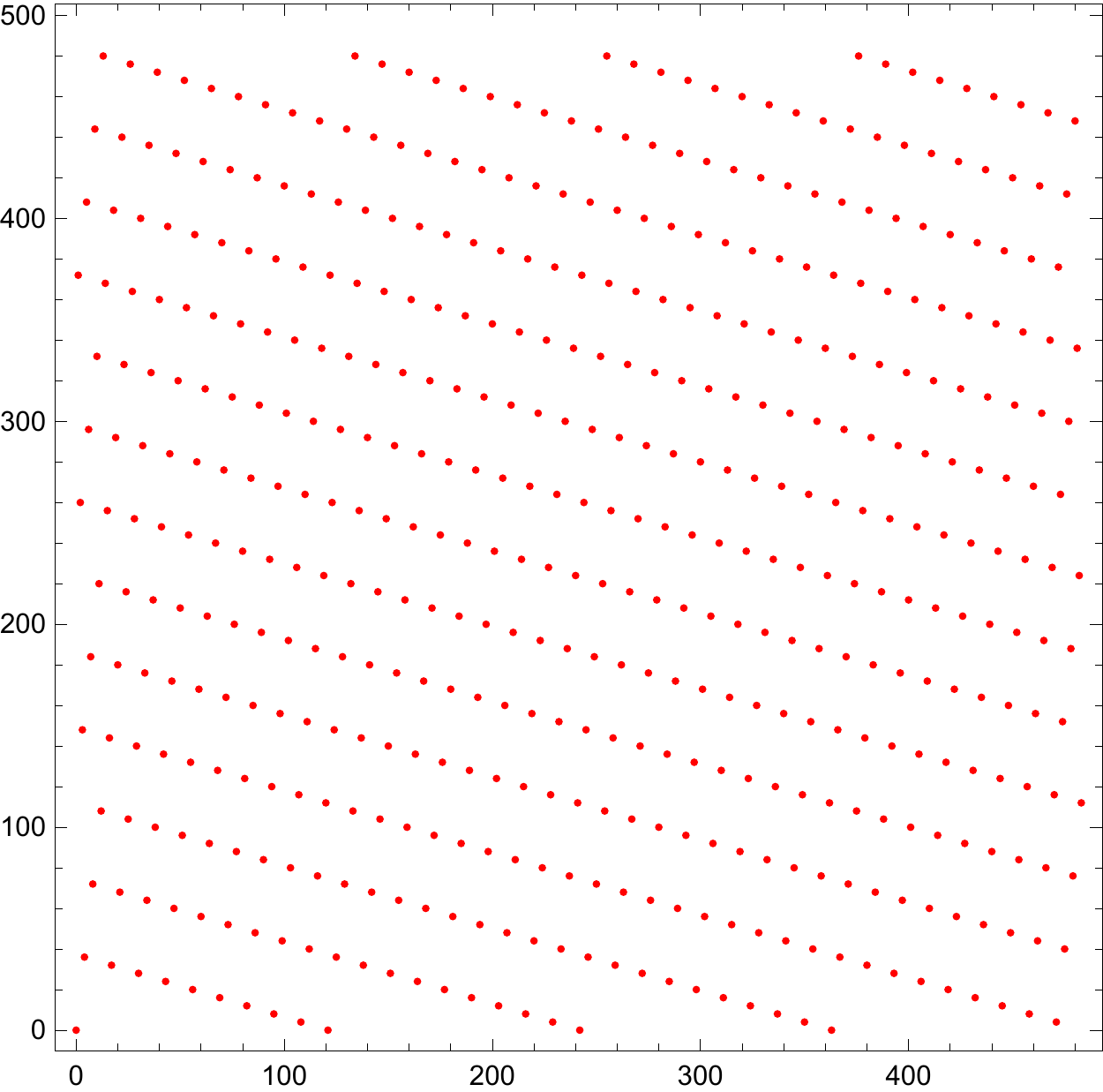}
}
\subfigure{
\includegraphics[width=0.31\textwidth]{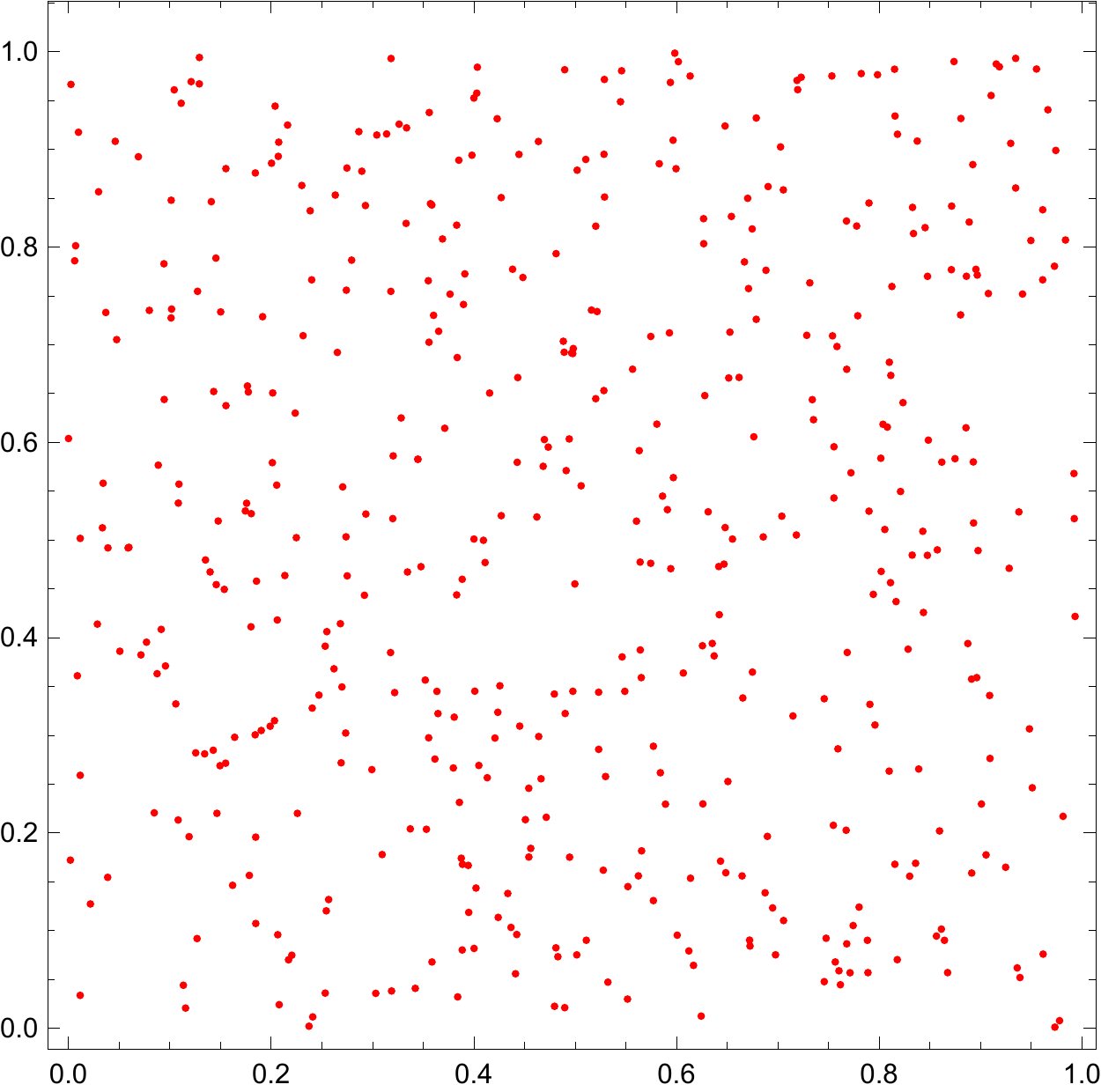}
}

\caption{The point set $K_{\alpha,\beta}(484)$ with  $\alpha=\sqrt[3]{91} \pmod{1}$ and $\beta=\sqrt[3]{91^2} \pmod{1}$ (\emph{left}), the modular lattice $L_{484,241,112}$ (\emph{middle}) and a set of $484$ random points (\emph{right}).} \label{fig:pointsets}
\end{figure}
\end{center}

\subsection{Results}
The aim of this note is to calculate bounds for the TSPs of modular lattices and Kronecker point sets.
In a first step, we show that the length of a shortest path in a modular lattice, $\Mlattice$, is intimately related to its shortest distance, $\lambda(\Mlattice)$:
\begin{thm} \label{thm1}
Let $\mathrm{gcd}(N,a,b)=1$.
If $\lambda(\Mlattice)$ denotes the shortest distance between points of the modular lattice $\Mlattice \subseteq [0,N-1]^2$ with $N$ points, then
\begin{equation*}
\frac{\lambda(\Mlattice)}{\sqrt{N}} \leq  \frac{\mathcal{L}(\Mlattice)}{N \sqrt{N}} \leq \frac{\lambda(\Mlattice) }{\sqrt{N}} + \frac{2 \sqrt{2}}{\sqrt{N}}.
\end{equation*}
\end{thm}
Fixing a lattice, we can explicitly compute our bounds using a result of Eisenbrand \cite{Eis01} who gave a fast and simple method for computing shortest vectors of lattices. Thus, it is natural to ask: Which modular lattice gives the largest constant? We define 
\begin{equation*}
f(N,a,b):=\lambda(\Mlattice)/\sqrt{N}.
\end{equation*}
The following construction shows how to get a lattice with a constant close to 1 for every integer $N$.
\begin{thm} \label{thm1a}
Let $N$ be a positive integer. If $a=1$ and $b=\lfloor \sqrt{N} \rfloor-1$, then
$$f(N,a,b) = \frac{\sqrt{(\lfloor \sqrt{N} \rfloor-1)^2 + 1 } }{\sqrt{N}}. $$
\end{thm}
In fact we can show even more; i.e. the constant in modular lattices is in general not upper bounded by 1. 
\begin{thm} \label{thm1b}
For infinitely many $N \in \NN$ there exists a pair of integers $(a,b)$ such that
$$ f(N,a,b)> 1. $$
\end{thm}

\begin{rem}
The point sets we study are the prime examples of \emph{low discrepancy} point sets which are widely used in numerical integration; see \cite{KN,N}. Steele \cite{Ste80} and Steinerberger \cite{Ste10} connected the uniform distribution properties of a point set to the length of the shortest path through the points by providing bounds for the traveling salesman path in terms of the \emph{discrepancy} of the point set. 
The lattices with the smallest discrepancy are those generated from a prime $N$ with $a=1$ and $1\leq b \leq N-1$ such that the continued fraction expansion of $b/N$ has the smallest possible partial quotients \cite{KN,Pau}.
Interestingly, our results show that the point sets with the smallest discrepancy are in general \emph{not} those with the longest shortest path for a fixed $N$ and $1\leq b \leq N-1$! 
\end{rem}

In a second step, we use classical results from the field of diophantine approximation to relate Kronecker point sets to modular lattices. That is, if $(p_i/q_i,r_i/q_i)$ is the $i$-th convergent of $(\alpha,\beta)$ for irrationals $0<\alpha,\beta<1$,
then the shortest path through the points of the corresponding Kronecker point set $\KsetQ$ with $N_i=q_i$ elements can be accurately approximated via the shortest path through the modular lattice $L_{N_i,p_i,r_i}$:
\begin{thm} \label{thm2}
Let $\alpha, \beta \in (0,1)$ be irrationals such that $1,\alpha, \beta$ are linearly independent over $\QQ$ and let $(p_i/q_i, r_i/q_i)$ be the $i$-th convergent of $(\alpha,\beta)$. If $N_i=q_i$ and $\KsetQ \subseteq [0,1]^2$ denotes the corresponding Kronecker set, then
\begin{align*}
\frac {\lambda(\MlatticeQ)}{\sqrt{N_i}} \left( 1 - \frac{5}{N_i^{\delta}} \right) \leq
\frac{\mathcal{L}( \KsetQ )}{\sqrt{N_i}} \leq
\frac{ \lambda(\MlatticeQ)}{\sqrt{N_i}} \left( 1+  \frac{3}{N_i^{\delta} } \right ) + \frac{2 \sqrt{2}}{\sqrt{N_i}}, 
\end{align*}
for a constant $\delta=\delta(\alpha,\beta)>0$.
\end{thm}

\begin{rem}
We double checked our bounds with the lengths of the shortest paths computed via the built-in function $\mathit{FindShortestTour}$ of the computer algebra system $\mathit{Mathematica}$. The numbers agree well in general, however, we noticed that the algorithm underlying $\mathit{FindShortestTour}$ sometimes inserts unnecessary line segments between lines parallel to the shortest vector of a lattice, thus giving slightly longer paths in these cases.
\end{rem}

\subsection{Additional remarks and open questions.} 
We have calculated the constants of modular lattices for all $2\leq N \leq 750$ and all pairs $(a,b)$ with $1\leq a,b \leq N-1$ and $\mathrm{gcd}(N,a,b)=1$. We plot $f_{\max}(N):= \max_{1 \leq a,b \leq N} f(N,a,b)$ in Figure \ref{sup} and note that the largest value we obtained is $f(209,1,56)=1.07383$.
We can compare this observation with the result of Karloff \cite{Kar89} who determined a general upper bound of $1.39159 \sqrt{N} + 11$ for the length of the traveling salesman tour through any set of $N$ points in $[0,1]^2$. If we plug our numerical results into the upper bound from Theorem \ref{thm1} we improve Karloff's bound for all modular lattices with $N\leq 750$. We are immediately led to ask:
\begin{qu}
Is there a modular lattice with $f(N,a,b)>f(209,1,56)$?
Is there a modular lattice $\Mlattice$ that maximizes $f(N,a,b)$? 
\end{qu}
Furthermore, Figure \ref{sup} suggests two potential strengthenings of Theorem \ref{thm1b}:
\begin{qu}
Is there a modular lattice $\Mlattice$ with $f(N,a,b)>1$ for every (prime) $N>N_0$?
\end{qu}
\begin{qu}
Is there an absolute $\varepsilon>0$ such that $f_{\max}(N)>1+\varepsilon$ for infinitely many $N>N_0$?
\end{qu}
Next, it is interesting to note that the sequence 
\begin{equation} \label{limit}
\left( f(N_i, p_i, r_i) \right)_{i\geq 1}
\end{equation}
of values that approximate the length of the TSP of $\KsetQ$
does not seem to converge to a limit. We illustrate this oscillating behavior in Example \ref{example2} and suspect it to be the generic case. 
\begin{pr}\label{prop1} Prove or disprove the existence of a limit for the sequence defined in \eqref{limit}. \end{pr}
We remark, that such an oscillating behavior was already observed by Platzman and Bartholdi \cite{Pla89} and Gao and Steele \cite{Gao94} in the context of the \emph{spacefilling curve heuristic} and of course in the above mentioned context of stationary ergodic processes \cite{AS15}.
Finally, it is natural to ask
\begin{qu} Are there other families of point sets for which precise asymptotic results can be obtained?\end{qu}
Interesting candidates could be jittered sampling sets \cite{PS16} or the well-known Hammersley resp. Halton point sets \cite{N,Pau}.

\begin{center}
\begin{figure}[h!]
\includegraphics[scale=0.7]{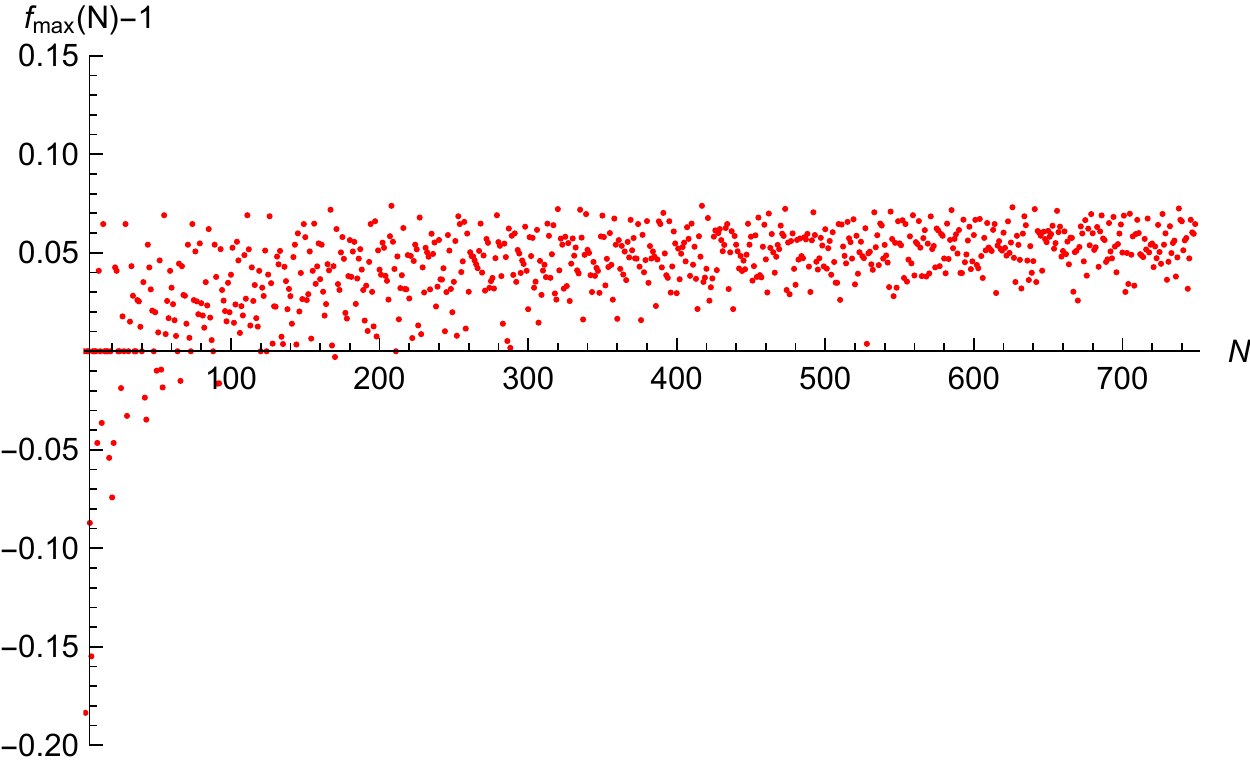}

\caption{ Plot of $f_{\max}(N)-1$ for $2\leq N \leq 750$.} \label{sup}
\end{figure}
\end{center}

\subsection{Outline} In Section \ref{sec2} we recall important facts about shortest vectors of lattices as well as rational approximations of (irrational) vectors in $\RR^2$. Section \ref{sec3} contains our results for modular lattices. Finally, we prove Theorem \ref{thm2} in Section \ref{sec41} and illustrate Problem \ref{prop1} with an example in Section \ref{example2}.

\section{Background}
\label{sec2}

\subsection{Rational approximations and continued fractions}
We refer to the book of Schweiger \cite{Sch} for a thorough introduction to the theory of multi-dimensional continued fractions and to the survey of Berth\'{e} \cite[Section 4]{Ber10} for a recent and well written overview.
The classical continued fraction algorithm produces, for every irrational $\alpha \in \RR$, a sequence of rational numbers $p_i/q_i$ that approximate $\alpha$ up to an error of order $1/q_i^2$. 
In the course of this paper, we need a multi-dimensional analogue that allows us to approximate a pair of irrationals $(\alpha,\beta) \in (0,1)^2$ by rational vectors.
There are several Jacobi-Perron type algorithms at our disposal; see \cite{Mee99, Schr98, Sch96}.
Each of these algorithms takes a pair of independent irrationals $(\alpha, \beta)$ as an input and outputs a sequence of integer triples $(q_i,p_i,r_i)$ called convergents, such that
\begin{equation*}
\max( |\alpha-p_i/q_i|, |\beta - r_i/q_i| ) \rightarrow 0, \text{ as } i \rightarrow \infty,
\end{equation*} 
which is known as \emph{weak convergence} of the algorithm. In fact, an even stronger result holds for these algorithms:
\begin{thm}\label{thm:delta}
There exists a constant $1/2\geq \delta > 0$ such that for almost every pair of numbers $(\alpha,\beta) \in [0,1]^2$ there exists $i_0=i_0(\alpha,\beta)$ such that for any $i>i_0$,
\begin{equation*}
\left | \alpha - \frac{p_i}{q_i} \right | < \frac{1}{q_i^{1+\delta}}, \left | \beta - \frac{r_i}{q_i} \right | < \frac{1}{q_i^{1+\delta}},
\end{equation*}
where $(q_i,p_i,r_i)$ is the $i$-th convergent of $(\alpha,\beta)$.
\end{thm}
This stronger convergence property is also referred to as \emph{strong convergence}. 
Furthermore, Meester points out in the last paragraph of \cite{Mee99} how to calculate explicit values for $\delta=\delta(\alpha,\beta)$ and he states the exact behavior of the $q_i$ on the exponential scale \cite[Corollary 1]{Mee99}. In general, it is known from a classical result of Perron \cite{Per21} that the optimal exponent of convergence any approximation algorithm can achieve is $1+1/d$, in which $d$ denotes the dimension. Thus, $\delta=1/2$ is the maximum we can hope for in two dimensions. However, there is no canonical algorithm that works for all pairs $(\alpha,\beta)$ equally well.

\subsection{Computing shortest Vectors}
Given a lattice $L=L((a,b),(m,0),(0,m))$ generated by the integer linear combinations of the vectors $(a,b),(m,0),(0,m)$, we can find a minimal set of generators, or basis, following \cite[Section 3]{Rot97}. 
A reduced basis for a lattice (in the sense of Minkowski) consists of two vectors $x_1$, $x_2$ with the following properties:
\begin{itemize}
\item $x_1$ is a shortest non-zero lattice vector.
\item $x_2$ is a shortest vector among the lattice vectors which are not parallel to $x_1$.
\end{itemize}
Lattice basis reduction is an important technique in computer science with an abundance of applications. Already Gauss invented an algorithm that finds a reduced basis of a two-dimensional integral lattice. That is, the algorithm takes a basis as an input and outputs a new basis with potentially shorter basis vectors; see \cite{Eis01}.

\section{Results for modular lattices}
\label{sec3}

In this section we prove our results for modular lattices. In Section \ref{sec31} we prove Theorem \ref{thm1}. In Section \ref{sec32} we illustrate our method and compute explicit bounds for particular lattices, thus proving Theorem \ref{thm1a} and Theorem \ref{thm1b}

\subsection{General result} \label{sec31}
In the following let $\mathrm{gcd}(N,a,b)=1$.
We start with an observation about shortest distances in modular lattices.

\begin{lem} \label{lem:help}
If $\Mlattice$ is a modular lattice, then $\lambda(\Mlattice) \leq 3/2 \sqrt{N}$.
\end{lem}

\begin{proof}
Let $v$ be a shortest vector of $L((a,b),(N,0),(0,N)) \subset [0,N-1]^2$ and assume that $\|v\| = x$. Then by linearity there is always at least a sixth of an open disc of radius $x$ attached to every point in the lattice that contains no other point; see Figure \ref{fig:short} (left).
There are $N$ points in the modular lattice such that we get (ignoring boundary effects) the following condition on $x$:
\begin{equation*}
N \cdot \frac{x^2 \pi}{6} \leq (N-1)^2.
\end{equation*}
This implies that $x\leq 3/2 \sqrt{N}$.
\end{proof}

\begin{center}
\begin{figure}[h!]

\subfigure{
\begin{tikzpicture}[scale=0.2]
\draw[thick] (0,6)--(0,0);
\draw[thick] (5,3)--(0,0);
\draw[dashed] (0,6)--(5,3);

\draw[dashed] (5,3) -- (5,15);
\draw[dashed] (10,0) -- (10,12);
\draw[dashed] (0,6) -- (0,12);
\draw[dashed] (0,6) -- (10,12);
\draw[dashed] (5,3) -- (10,6);

\node at (0,0) {$\bullet$}; 
\node at (0,6) {$\bullet$}; 
\node at (0,12) {$\bullet$};
\node at (5,3) {$\bullet$}; 
\node at (5,9) {$\bullet$}; 
\node at (5,15) {$\bullet$}; 
\node at (10,0) {$\bullet$}; 
\node at (10,6) {$\bullet$}; 
\node at (10,12) {$\bullet$}; 

\draw [very thick](0,6) arc (90:30:6cm);
\draw [dashed](0,12) arc (90:30:6cm);
\draw [dashed](5,3) arc (90:30:6cm);
\draw [dashed](5,9) arc (90:30:6cm);
\draw [dashed](5,15) arc (90:30:6cm);

\draw [dotted, thick](0,2) arc (90:30:2cm);

\node at (-1,3) {\small $x$}; 
\node at (3,0) {\small $x$}; 
\node at (0.6,1.1) {\tiny $\alpha$};
\node at (1.8,3.5) {\tiny $d$};

\end{tikzpicture}

} \qquad \qquad \qquad
\subfigure{
\begin{tikzpicture}[scale=0.4]
\draw[thick] (0,6)--(0,0);
\draw[thick] (5,3)--(5,9);
\draw[dashed] (0,3) -- (5,3);
\draw[dashed] (5,0) -- (5,3);
\draw[dotted] (0,0)--(5,3);
\draw[dotted] (0,6)--(5,3);
\node at (0,0) {$\bullet$}; 
\node at (0,6) {$\bullet$}; 
\node at (5,3) {$\bullet$}; 
\node at (5,9) {$\bullet$}; 

\node at (5,0) {$\circ$}; 
\node at (5,6) {$\circ$};

\node at (-1,0) {\small $x_1^h$}; 
\node at (-1,6) {\small $x_2^h$}; 
\node at (6.2,3) {\small $x_1^{h+1}$}; 
\node at (-1,3) {\small $x$}; 
\node at (2.5,1) {\small $d$}; 

\end{tikzpicture}
}

\caption{\emph{Left:} An extremal configuration. The minimal distance between any two points in the lattice is $x=\|v\|$. Hence, $d\geq x$ such that $\alpha \in [\pi/3, \pi/2]$ and there is at least a sixth of an open disc attached to each point that contains no other point of the lattice.
\emph{Right: } The distance $d$ between points on neighboring lines.} \label{fig:short}
\end{figure}
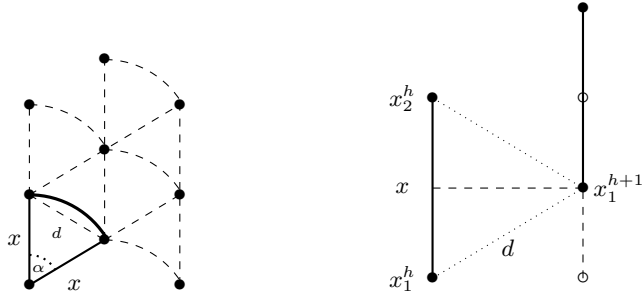
\end{center}

We remark that Lemma \ref{lem:help} implies that the points in any modular lattice $\Mlattice$ can be partitioned into at most $k=k(\Mlattice) \in \mathcal{O}(\sqrt{N})$ parallel lines such that the distance of neighboring points on a line is $\lambda(\Mlattice)$; compare with the image in the middle of Figure \ref{fig:pointsets}. In particular, we always have $$ k(\Mlattice) \leq 2 \lambda(\Mlattice). $$

\begin{lem} \label{lem1}
Let $\Mlattice$ be a modular lattice, then
\begin{equation*}
\frac{\lambda(\Mlattice)}{\sqrt{N}} \leq  \frac{\mathcal{L}(\Mlattice)}{N \sqrt{N}} \leq \frac{\lambda(\Mlattice) }{\sqrt{N}} + \frac{2 \sqrt{2}}{\sqrt{N}}.
\end{equation*}
\end{lem}

\begin{proof}
We begin with the upper bound and observe that we can build a valid path out of $N$ line segments by first connecting all points that lie on lines parallel to the shortest vector $v$ of $L((a,b),(N,0),(0,N))$ and then connect the resulting $k$ lines to obtain a closed path. From Lemma \ref{lem:help} we know that $\|v \|, k \in \mathcal{O}(\sqrt{N})$.
Furthermore, we observe that two lines can have at most a distance of $N\sqrt{2}/k$ such that the minimal distance $d$ of two points on these lines is upper bounded by $N\sqrt{2}/k + \|v\|$; see Figure \ref{fig:short} (right). 
This suffices to upper bound the length of the shortest path. Note that we divide our estimate by $N$ to scale $[0,N-1]^2$ to $[0,1]^2$. We obtain for $\lambda = \lambda(\Mlattice)$:
\begin{equation*}
\frac{\mathcal{L}(\Mlattice)}{N \sqrt{N}} \leq \frac{(N-k) \lambda + k (N\sqrt{2}/k+ \lambda ) + N \sqrt{2}}{N \sqrt{N}} = \frac{\lambda}{\sqrt{N} } + \frac{\sqrt{2} + \sqrt{2}}{\sqrt{N}}.
\end{equation*}
As for the lower bound, we simply assume that all $N$ line segments are of minimal length $\|v\|$ which implies that the shortest path has at least length $N \cdot \lambda(\Mlattice)$.
\end{proof}

\subsection{Long shortest vectors} \label{sec32}
Now we look at particular lattices. 
For given integers $N, b$ with $1<b<N$ and $\mathrm{gcd}(N,b)=1$, there are unique positive integers $x, y$ with $y<b$ and $N=b \cdot x + y$.
According to our definition, each $\MMlattice$, interpreted as point set, contains the triangle $ABC$ with $A=(0,0)$, $B=(1,b)$ and $C=(x+1,b-y)$; for an illustration see Figure \ref{fig:triangle}. Note that $AB=(1,b)$, $AC=(x+1, b-y)$ and $BC=(x,-y)$.

\begin{center}
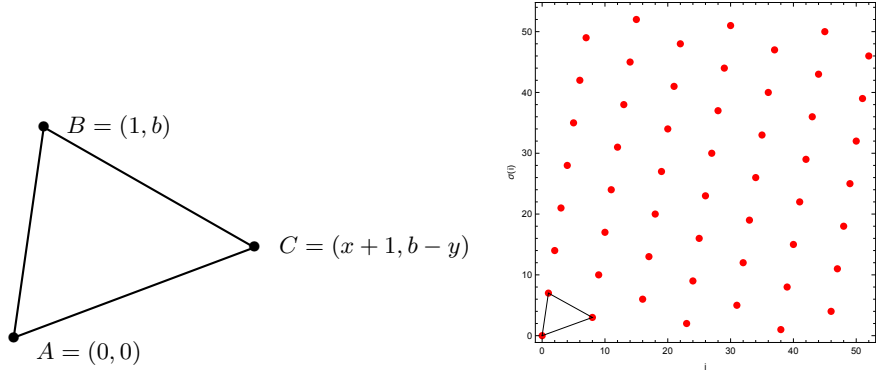
\begin{figure}[h!]
\subfigure{
\begin{tikzpicture}[scale=0.4]
\draw[thick] (0,0)--(1,7)--(8,3)--(0,0);
\node at (0,0) {$\bullet$}; 
\node at (1,7) {$\bullet$}; 
\node at (8,3) {$\bullet$}; 

\node at (2.5,-0.5) {\small $A=(0,0)$}; 
\node at (3.5,7) {\small $B=(1,b)$}; 
\node at (12,3) {\small $C=(x+1,b-y)$}; 
\end{tikzpicture}
}
\subfigure{
\includegraphics[width=0.33\textwidth]{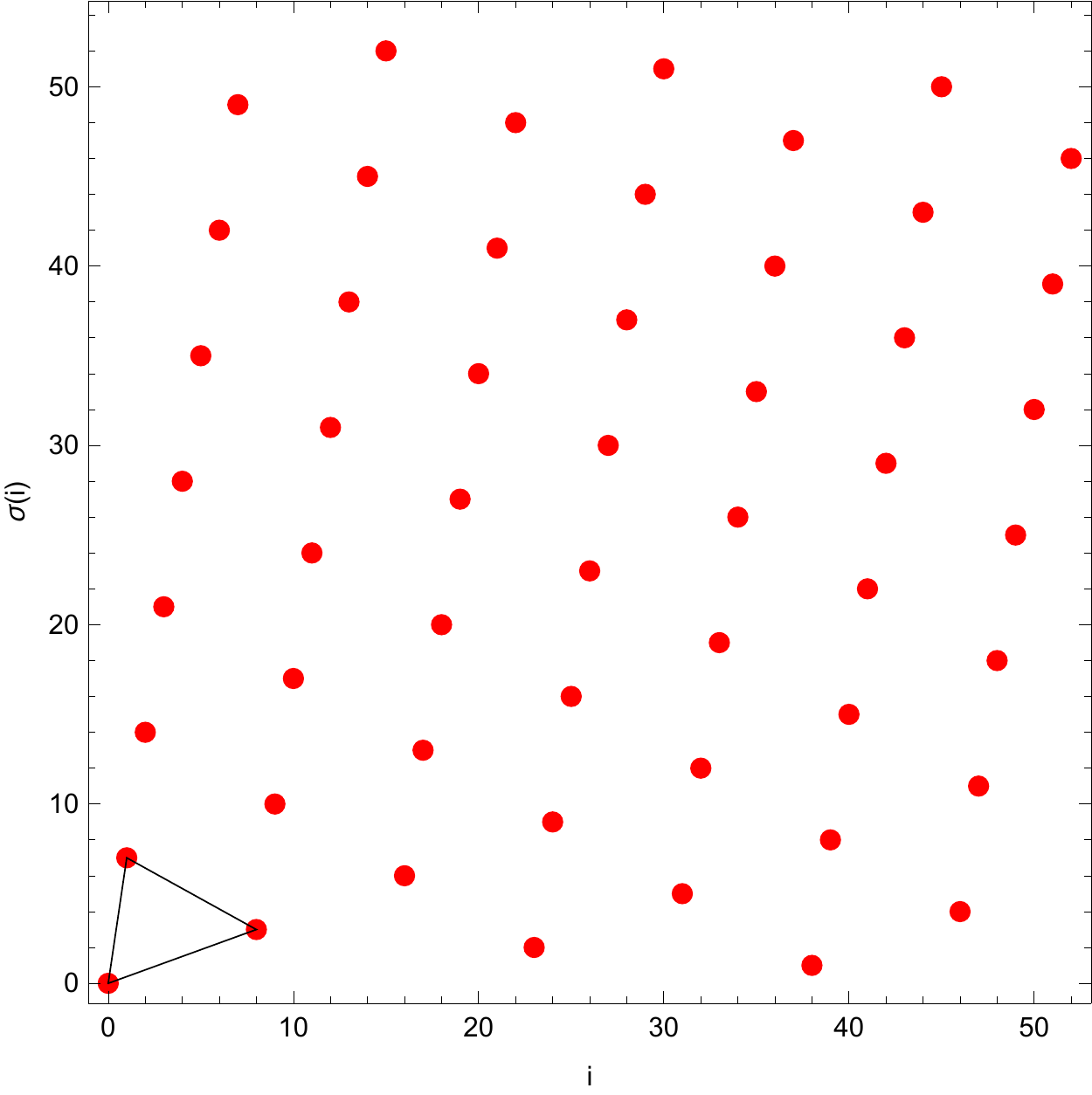}
}
\caption{\emph{Left:} The triangle $ABC$. \emph{Right:} $L_{53,1,7}$ and the corresponding triangle.} \label{fig:triangle}
\end{figure}
\end{center}

In the following, we first relate the shortest vector of $\MMlattice$ to the triangle $ABC$, before we use Lemma \ref{lem:example} to characterize a particular family of such modular lattices.

\begin{lem} \label{lem:example}
Let $\MMlattice$ be a modular lattice for an integer $N$ with $\mathrm{gcd}(N,b)=1$ and $1<b<N$.
If $AB$ is the shortest edge of the triangle $ABC$, then $(1,b)$ is a shortest vector of $\MMlattice$. In particular, we have
{\small
\begin{equation*}
\frac{\sqrt{1+b^2}}{\sqrt{N}} + \frac{b (\min\{ \| AC\|, \| BC\| \} -\| AB\| ) }{N \sqrt{N}} \leq  \frac{\mathcal{L}(\MMlattice)}{N \sqrt{N}} \leq \frac{\sqrt{1+b^2}}{\sqrt{N}} + \frac{b (\max \{ \| AC\|, \| BC\| \} - \| AB\|  )}{N \sqrt{N} } + \frac{\sqrt{2}}{\sqrt{N}}.
\end{equation*}
}
\end{lem}

\begin{proof}
We start with the lattice $L=((1,b),(N,0),(0,N))$ and use the method outlined in \cite[Section 3]{Rot97} to compute a basis of this lattice. In particular, we obtain
\begin{equation}
b_1=(1, (1-N)b ) \ \ \ \text{ and } \ \ \ b_2=(0,N).
\end{equation}
In a second step, we explicitly reduce this basis using the algorithm of Gauss outlined in \cite{Eis01}.
That is, we first replace $b_1$ by $(1,b)$ since $(1,b)=b_1+ b b_2$ which is clearly shorter than $b_1$. In a second reduction step we reduce $b_2$ either to $(-x,y)=-BC$ or $(-x-1,y-b)=-AC$ depending on which vector is shorter. By the assumption that $AB$ is shortest in the triangle $ABC$ it follows that both vectors are longer than $(1,b)=AB$. Moreover, it is easy to see that there is no further reduction. Hence, the reduced basis of the lattice is either $AB$ and $AC$ or $AB$ and $BC$. If follows from \cite[Lemma 2]{Rot97} that $\min(AC,BC)$ is shortest among all vectors in the lattice that are not parallel to $AB$.\\
Knowing that $(1,b)$ is a shortest vector, it is now easy to get bounds for the length of the shortest path through all points:
Since $N=b \cdot x+ y$ we have $b$ lines each containing either $x$ or $x+1$ points from which we obtain the stated lower bound.\\
For the upper bound it suffices to explicitly construct a valid path. We achieve this by first connecting all the points on the $b$ parallel lines. Next, we connect neighboring lines and observe that the required $(b-1)$ line segments  each have a length of at most $\max(AC,BC)$. Finally, we close the path by adding a line segment of length at most $N\sqrt{2}$. 
\end{proof}

Thus, it follows that if we choose $b$ to be roughly of size $\sqrt{N}$ and such that $AB$ is the shortest edge of $ABC$, then we obtain a point set with long TSP. The next lemma makes this observation precise and characterizes one such family of lattices.

\begin{lem}
For an integer $N$ set $b=\lfloor \sqrt{N} \rfloor -1$. Then the edge $AB$ is the shortest edge of $ABC$ in $\MMlattice$.
\end{lem}

\begin{proof}
Recall that there are unique integers $x,y$ with $b>y$ such that $N=b\cdot x+y$.
Setting $b=\lfloor \sqrt{N} \rfloor -1$ implies that $x\geq \lceil \sqrt{N} \rceil$. To see this set $\lfloor \sqrt{N} \rfloor = \sqrt{N}-z$ for a $z \in [0,1)$. Then
\begin{equation*}
b\cdot x = (\lfloor \sqrt{N} \rfloor -1)\lceil \sqrt{N} \rceil=(\sqrt{N}-z-1)(\sqrt{N}+(1-z) ) = N - 2z \sqrt{N} -1 +z^2 \leq N.
\end{equation*}
Hence,
\begin{align*}
\|AB\| &= \sqrt{1+b^2} \leq \sqrt{(\sqrt{N}-1)^2 + 1} \leq \sqrt{N},\\
\|BC\| &= \sqrt{x^2 + y^2} \geq \sqrt{x^2 +1} > \sqrt{N +1},\\
\|AC\| &= \sqrt{(x+1)^2 + (b-y)^2} \geq \sqrt{(x+1)^2} > \sqrt{N +1}.
\end{align*}
\end{proof}
Thus, we get
\begin{equation*}
\underset{N\rightarrow \infty}\lim \frac{ \mathcal{L}(L_{N,1,\lfloor \sqrt{N} \rfloor -1}) }{N\sqrt{N}} = \underset{N\rightarrow \infty}\lim \frac{\sqrt{(\lfloor \sqrt{N} \rfloor-1)^2 + 1 } }{\sqrt{N}} = 1.
\end{equation*}

\begin{ex}
If $N=479$, then $b=20$ and $479=20 \cdot 23+19$ with $AB=(1,20)$, $AC=(24,1)$ and $BC=(23,-19)$. Thus, $AB$ is indeed the shortest edge of $ABC$ and by Lemma \ref{lem:example} a shortest vector of $L_{479,1,20}$. Plugging the according values into our formula, we obtain:
$$ 0.9225\ldots \leq \frac{\mathcal{L}(L_{479,1,20})}{\sqrt{479}} \leq 0.9982\ldots $$
\end{ex}

Interestingly, we can do even better for infinitely many integers $N$ as the following result shows.
\begin{lem}
For infinitely many integers $N$ there exists a pair $(a,b)$ such that $f(N,a,b)>1$.
\end{lem}

\begin{proof}
For a given $N$ we set $a=1$ and $b=\lceil \sqrt{N} \rceil = \sqrt{N}+(1-\{ \sqrt{N} \})$ and recall from Lemma \ref{lem:example} that if $AB=(1,b)$ is a shortest edge of $ABC$, then $AB$ is a shortest vector of $L_{N,1,b}$.
Again, there are unique positive integers $x,y$ with $x,y < b$ such that $N=bx+y$. In particular we have for $1/2 > z^*(N):=1/2 (1 + 2 \sqrt{N} - \sqrt{1 + 4 N})$ that
$$ x = \begin{cases} 
b-1 &\mbox{if } \{ \sqrt{N} \} \geq z^*(N)\\ 
b-2 & \mbox{otherwise.}
\end{cases} 
$$
Since the infinite sequence $(\sqrt{N})_{N\geq 1}$ is uniformly distributed modulo 1 (see \cite[Chapter 1]{KN}), we find infinitely many integers satisfying the first case with the even stronger condition $\{ \sqrt{N} \}>1/2$.  \\
We choose such an integer.
Hence, $x=b-1$ such that $AB=(1,b)$, $AC=(x+1,b-y)=(b,b-y)$ and $BC=(x,-y)=(b-1,-y)$. Since $b>y$ we trivially have $\|AB\| \leq \|AC\|$. To obtain $\|AB\| < \|BC\|$ as well, we need that $b^2 +1 < (b-1)^2 + y^2$ or, simplified, that
\begin{equation}\label{cond1}
2b < y^2.
\end{equation}
Set $z'=1-\{ \sqrt{N} \} \in [0,1]$, then by assumption $z'<1/2$ and 
$$b=\sqrt{N}+z', \ \ \  y=N-bx = N-(\sqrt{N}+z')(\sqrt{N}+z'-1).$$ 
Define 
$$g(N,z):= (N-(\sqrt{N}+z)(\sqrt{N}+z-1))^2 - 2 (\sqrt{N}+z),$$ 
which is a polynomial of degree $4$ in $z$. We observe that for $N>7$, $g$ has two real roots and $g(N,0), g(N,1) > 0$ and $g(N,0.5) < 0$. This implies that there is a root $\rho(N) \in (0,1/2)$. Asymptotically we find 
\begin{equation}\label{limit1}
\lim_{N\rightarrow \infty} \rho(N) = 1/2.
\end{equation} 
Hence, if we pick $N$ such that $z' \in (0,\rho(N))$, then $g(N,z')>0$ and therefore \eqref{cond1} holds. Moreover, from \eqref{limit1} it is clear that there is an index $N_0$ such that for given $\varepsilon \in (1/2,1)$ and $N>N_0$, $\rho(N)>1- \varepsilon$. Thus, for all $N>N_0$ with $\{ \sqrt{N} \} \in (\varepsilon, 1)$ the inequality in \eqref{cond1} holds. Again, since the sequence $(\sqrt{N})_{N\geq 1}$ is uniformly distributed modulo 1, there are infinitely many such $N$.
\end{proof}

\begin{rem}
For $\varepsilon=3/4$ we get $N_0=87$. Thus for all $N>87$ with $\{ \sqrt{N} \} \in (3/4,1)$ the inequality in \eqref{cond1} holds.
\end{rem}

\section{Results for 2-dimensional Kronecker sequences}
\label{sec4}

\subsection{Relation to modular lattices} \label{sec41}
Let $\alpha, \beta \in (0,1)$ be irrationales such that $1,\alpha, \beta$ are linearly independent over $\QQ$. As we have seen in Section \ref{sec2} we can use the Jacobi-Perron (or one of the related algorithms) to approximate almost all pairs $(\alpha, \beta)$ by triples $(q_i, p_i, r_i)$ such that
\begin{equation} \label{assump}
\left | \alpha - \frac{p_i}{q_i} \right | < \frac{1}{q_i^{1+\delta}} \text{ and } \left | \beta - \frac{r_i}{q_i} \right | < \frac{1}{q_i^{1+\delta}},
\end{equation}
for a constant $\delta>0$ and all indices $i>n_0$.
Combining this result with our results of the previous section, we can approximate $\KsetQ$ with the lattice $\MlatticeQ$. The following lemma makes the relation between the corresponding shortest paths precise.

\begin{lem}
Let $(p_i/q_i, r_i/q_i)$ be the $i$-th convergent of $(\alpha,\beta)$ such that \eqref{assump} holds. If $N_i=q_i$ and $\lambda=\lambda(\MlatticeQ)$ is the shortest distance in $\MlatticeQ$ then
\begin{align*}
\frac {\lambda}{\sqrt{N_i}} \left( 1- \frac{5}{N_i^{\delta}} \right) \leq
\frac{\mathcal{L}( \KsetQ )}{\sqrt{N_i}} \leq
\frac{ \lambda}{\sqrt{N_i}} \left( 1+ \frac{3}{N_i^{\delta} } \right) + \frac{2 \sqrt{2}}{\sqrt{N_i}}, 
\end{align*}
for a constants $\delta=\delta(\alpha,\beta)>0$.
\end{lem}

\begin{proof}
To prove the lemma, we approximate a path through all points of $\KsetQ$ by estimating its difference to the path in $\MlatticeQ$ which we build out of the shortest vector $v_i$ of $L((p_i,r_i),(q_i,0),(0,q_i))$.
First, we normalize the points in $\MlatticeQ$ by multiplying with $1/q_i$ to obtain a point set in $[0,1]^2$ and we write $w_i:=v_i/q_i$ for the corresponding shortest vector.
Given two points $x_0=n_0(p_i/q_i, r_i/q_i), x_1=n_1(p_i/q_i, r_i/q_i) \in \MlatticeQ$ such that $x_1=x_0+w_i$, we observe that there are related points $x^*_0=n_0(\alpha,\beta), x^*_1=n_1(\alpha,\beta) \in \KsetQ$ with $x^*_1=x^*_0+w_{0,1}^*$. Setting $d=(\alpha - p_i/q_i, \beta - r_i/q_i)$, we can write 
\begin{align} \label{estimate}
w_{0,1}^*=(x_1+n_1 d ) - (x_0+n_0 d ) =x_1-x_0 + (n_1-n_0) d  = w_i + (n_1-n_0) d ;
\end{align}
see Figure \ref{fig:bound} for an illustration.
Consequently,
\begin{equation*} 
\|w_{0,1}^*\| = \| w_i + (n_1-n_0) d \|.
\end{equation*}
Recall that we have $k \in \mathcal{O}(\sqrt{q_i})$ parallel lines $l_1\ldots,l_h, \ldots l_k$ such that we can write $x_j^{h}=x_0^{h}+j v_i$, with $0\leq j \leq J$, for the points on a particular line.
Fixing a line $l_h$, we can now bound the sum of the lengths of the line segments through the points of $\KsetQ$ with the same indices. 
To get a lower bound, we use the triangle inequality:
\begin{align*}
\sum_{j=0}^{J_h-1} \| w_{j,j+1}^*\| &= \sum_{j=0}^{J_h-1} \| w_i + (n_{j+1}-n_j) d \| \\
&\geq \|  \sum_{j=0}^{J_h-1} w_i + (n_{j+1}-n_j) d \| = \| J_h w_i + (n_{J_h} - n_0) d \| \\
&\geq \| J_h w_i \| - | n_{J_h}-n_0 | \|d\|.
\end{align*}
Since we have $k$ such lines we lower bound the length of a path through the points of $\KsetQ$ as follows:
\begin{align*}
\mathcal{L}(\KsetQ) & \geq \sum_{h=1}^k \left(  \sum_{j=0}^{J_h-1} \| w_{j,j+1}^*\| \right) \geq \sum_{h=1}^k \left( \| J_h w_i \| - | n_{J_h}-n_0 | \|d\| \right) \\
&\geq (q_i - k) \|w_i\| - k q_i \|d\| = q_i \|w_i\| - k( \| w_i\| + q_i \|d\|)
\end{align*}
Dividing by $\sqrt{q_i}$ and setting $\|w_i\|=\|v_i\|/q_i $, we get
\begin{align*}
\frac{\mathcal{L}(\KsetQ)}{\sqrt{q_i}} &\geq \frac{q_i \|w_i\| - k( \| w_i\| + q_i \|d\|)}{\sqrt{q_i}} = \frac{\|v_i\|}{\sqrt{q_i}} - \frac{k}{\sqrt{q_i}} \left( \frac{\|v_i\|}{q_i}  + q_i \|d\| \right) \\
& > \frac{\|v_i\|}{\sqrt{q_i}} \left(  1 - 2 \left(  \frac{3 }{2 \sqrt{q_i}} + \frac{\sqrt{2}}{q_i^{\delta}} \right) \right),
\end{align*}
since $\| v_i\| \leq 3/2 \sqrt{q_i}$ and $q_i \|d\| = \| (q_i \alpha - p_i, q_i \beta - r_i) \| < \sqrt{2}/q_i^{\delta}$.

Rewriting \eqref{estimate} as $w_i = w_{0,1}^* + (n_1-n_0) (-1)d$, we can use the same arguments as for the lower bound to obtain
\begin{equation*}
J_h \|w_i\| + | n_{J_h}-n_0 | \|d\| \geq  \sum_{j=0}^{J_h-1} \| w_{j,j+1}^*\|.
\end{equation*}
Connecting all the $k$ lines yields
\begin{align*}
\mathcal{L}(\KsetQ) &\leq (q_i-k) \|w_i\| + k q_i \|d\| + k (\|w_i\| +\sqrt{2}/k) + \sqrt{2}\\
&= q_i \|w_i\| + k q_i \|d\| + \sqrt{2} + \sqrt{2},
\end{align*}
hence,
\begin{align*}
\frac{\mathcal{L}(\KsetQ)}{\sqrt{q_i}} &\leq \frac{ \|v_i\|}{\sqrt{q_i}} \left(1+ \frac{2 \sqrt{2} }{q_i^{\delta} } \right)+ \frac{2 \sqrt{2}}{\sqrt{q_i}},
\end{align*}
\end{proof}

\begin{center}
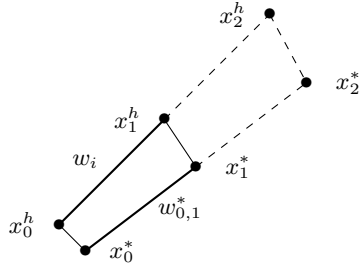
\begin{figure}[h!]
\begin{tikzpicture}[scale=0.7]

\node at (0,0) {$\bullet$}; 
\node at (2,2) {$\bullet$}; 
\node at (4,4) {$\bullet$}; 

\node at (0.5,-0.5) {$\bullet$}; 
\node at (2.6, 1.1) {$\bullet$}; 
\node at (4.7, 2.7) {$\bullet$}; 

\draw[thick] (0,0)--(2,2);
\draw[dashed] (2,2)--(4,4);

\draw[thick] (0.5,-0.5) -- (2.6,1.1);
\draw[dashed] (2.6,1.1) -- (4.7,2.7);
\draw (0,0) -- (0.5,-0.5);
\draw (2,2) -- (2.6,1.1);
\draw[dashed] (4,4) -- (4.7,2.7);

\node at (-0.7,0) {\small $x_0^h$}; 
\node at (1.3,2) {\small $x_1^h$}; 
\node at (3.3,4) {\small $x_2^h$}; 
\node at (1.2,-0.5) {\small $x_0^*$};
\node at (3.4,1.1) {\small $x_1^*$};
\node at (5.5,2.7) {\small $x_2^*$};

\node at (0.5,1.2) {\small $w_i$}; 
\node at (2.3, 0.3) {\small $w_{0,1}^*$}; 
\end{tikzpicture}
\caption{Illustration of our estimate of the difference in the length of the line segments in $\KsetQ$ and $\MlatticeQ$.} \label{fig:bound}
\end{figure}
\end{center}

\subsection{An example} \label{example2}
To illustrate our results we set $\alpha=\sqrt[3]{91} \pmod{1} = 0.4979\ldots$ and $\beta=\sqrt[3]{91^2} \pmod{1}=0.2314\ldots$. This choice is motivated by the work of Raju \cite{Raj76} who determined a family of irrationals such that the pair $(\omega, \omega^2)$ has a periodic Jacobi-Perron algorithm for every $\omega$ from this family. Going back to the seminal work of Perron \cite{Per07} it is known that pairs of irrationals with periodic Jacobi-Perron algorithm can be approximated such that $\delta=1/2$ in Theorem \ref{thm:delta}. \\

We see from our results that the gap between our lower and upper bounds is $\mathcal{O}(1/\sqrt{q_i})$. Therefore, we simply approximate the constants of $K_{\alpha,\beta}(q_i)$ resp. $L_{q_i,p_i,r_i}$ in the following by $f(q_i,p_i,r_i)$.
From the Jacobi-Perron algorithm we obtain the triples $(q_3,p_3,r_3)=(241,120,56)$, $(q_4,p_4,r_4)=(484,241,112)$ and $(q_5,p_5,r_5)=(972,484,225)$.
Using the standard basis reduction algorithm we obtain the lengths of the shortest vectors $\|v_3\|=\|(13,10)\|$, $\|v_4\|=\|(13,4)\|$ and $\|v_5\|=\|(28,9)\|$ and so on.
Interestingly, there is no obvious pattern or convergence in the sequence $\left( f(q_i,p_i,r_i) \right)_{i\geq 3}$ of constants; see Table \ref{table1}. Intuitively this makes sense, as it suggests that the Kronecker point sets do not converge to some kind of limit lattice. It would be interesting to further investigate this sequence and prove its non-convergence, which we suspect to be the generic behavior.

\begin{table}[h!] 
\begin{tabular}{ |c|c|c||c|c|c| }
\hline
 	$i$			&  		$q_i$	& $f(q_i,p_i,r_i)$		& 	$i$	&  		$q_i$	& $f(q_i,p_i,r_i)$		\\
\hline
$3$			&  		$241$	& $1.0055$		& 	$10$	&  		$28553528$	& $0.8961$		\\
$4$			&  		$484$	& $0.6182$		& 	$11$	&  		$57343144$	& $0.6323$		\\
$5$			&  		$972$	& $0.9433$		& 	$12$	&  		$3453390097$	& $0.2099$		\\
$6$			&  		$58537$	& $0.9544$		& 	$13$	&  	$6935333722$	& $0.3151$		\\
$7$			&  		$117558$	& $0.7122$		& 	$14$	&  	$13928010588$	& $0.2224$		\\
$8$			&  		$236088$	& $1.0002$		& 	$15$	&  	$838789966513$	& $0.5902$		\\
$9$			&  		$14217985$	& $0.5703$		& 	$16$	&  	$1684515266748$	& $0.8371$	\\
\hline
\end{tabular} 
\\[6pt]
\caption{The denominators of the $i$-th convergent to $(\alpha,\beta)$ and the approximated TSP constants.} \label{table1}
\end{table}

\section*{acknowledgements}
I would like to thank Stefan Steinerberger for bringing this research problem to my attention and for many interesting discussions. Furthermore, I would like to thank an anonymous referee for very constructive feedback and for simplifying my initial proof of Lemma \ref{lem:help}.



\end{document}